\documentclass[11pt,a4paper]{article}
\usepackage{latexsym}
\usepackage{amsmath}
\usepackage{amsthm}
\usepackage{amssymb}
\usepackage[dvips]{graphicx}
\usepackage{times}
\usepackage{amsmath, amsthm, amscd, amsfonts, amssymb, graphicx, color}
\usepackage[bookmarksnumbered, colorlinks, plainpages]{hyperref}
\usepackage{wrapfig}
\usepackage{subfig}
\begin{document}

\setlength{\topmargin}{0.0cm} \setlength{\textheight}{21.4cm}
\setlength{\textwidth}{15.6cm} \setlength{\oddsidemargin}{.4cm}
\setlength{\evensidemargin}{.4cm} \setlength{\parskip}{1ex}
\setlength{\parindent}{1cm}


\flushbottom
\newcommand{\bt}{\beta}
\newcommand{\al}{\alpha}
\newcommand{\laa}{\lambda_\alpha}
\newcommand{\lab}{\lambda_\beta}
\newcommand{\no}{|\Omega|}
\newcommand{\nd}{|D|}
\newcommand{\la}{\lambda}
\newcommand{\ro}{\varrho}
\newcommand{\cd}{\chi_{D}}
\newcommand{\cdc}{\chi_{D^c}}
\newcommand{\be}{\begin{equation}}
\newcommand{\ee}{\end{equation}}
\newcommand{\Om}{\Omega}
\newcommand{\h}{H^1_0(\Omega)}
\newcommand{\lt}{L^2(\Omega)}
\newtheorem{thm}{Theorem}[section]
\newtheorem{cor}[thm]{Corollary}
\newtheorem{lem}[thm]{Lemma}
\newtheorem{prop}[thm]{Proposition}
\theoremstyle{definition}
\newtheorem{defn}{Definition}[section]
\theoremstyle{remark}
\newtheorem{rem}{Remark}[section]
\numberwithin{equation}{section}
\renewcommand{\theequation}{\thesection.\arabic{equation}}
\numberwithin{equation}{section}

\title{ Extremal basic frequency of non-homogeneous plates}
\author{\bf Abbasali Mohammadi \\
\small  Department of Mathematics, College of Sciences,\\
\small Yasouj University, Yasouj, Iran, 75914-353\\
\small Email addresses: \\
\small mohammadi@yu.ac.ir,  abbasalimath@gmail.com }
\maketitle \hrule \textbf{Abstract.} In this paper we propose two numerical algorithms to derive the extremal principal eigenvalue
of the bi-Laplacian operator under Navier boundary conditions or Dirichlet boundary conditions. Consider a  non-homogeneous hinged or clamped plate $\Omega$,
the algorithms converge to the density functions on $\Omega$ which  yield the maximum or minimum basic frequency of the plate.
\hrule

\section{Introduction}\label{intro}

 Eigenvalue problems for elliptic partial differential equations  have many applications in
engineering and applied sciences and these problems have been
intensively attractive to mathematicians in the past decades \cite{henrot}.

 This paper is concerned with  a fourth-order  elliptic eigenvalue problem modeling  the vibration of a non-homogeneous plate $\Omega$  which is either hinged
or clamped  along the boundary $\partial \Omega$. Several materials with $m$ different kinds of  densities $0<c_1<c_2<...<c_m$  are given where the area of the domain with density $c_i$ is $S_i>0$, $i=1..m$. The problem involves  geometrical constraints that can be described as $\sum_{i=1}^{m}S_i$ should be equal to the area of $\Om$.
 We investigate the location of these materials throughout $\Omega$ in order to optimize  the basic frequency
 in the vibration of the corresponding plate.

  Motivated by  the above explanation, we introduce the mathematical equations governing the structure and associated optimization problems. Let $\Omega$ be a bounded smooth domain in $\mathbb{R}^N$ and let $ \rho_0(x)=c_1\chi_{D_1}+...+c_m \chi_{D_m}$, the density function, be a measurable function such that $|D_i|=S_i>0$, $(i=1..m)$ and $\sum_{i=1}^{m}S_i=|\Om|$   where  $|.|$ stands for Lebesgue measure. Define $\mathcal{P}$ as the family of all measurable functions which  are  rearrangement of
  $\rho_0$.  For $\rho \in \mathcal{P} $, consider eigenvalue problems

  \begin{equation}  \label{mpde1}
\Delta^2 u= \lambda \rho u ,\quad \mathrm{in}\quad\Omega, \quad u=0,\, \Delta u=0,  \qquad
\mathrm{on}\quad\partial \Omega,
\end{equation}
 \begin{equation}  \label{mpde2}
\Delta^2 v= \Lambda \rho v ,\quad \mathrm{in}\quad\Omega, \quad v=0,\, \frac{\partial v}{\partial n}=0,  \qquad
\mathrm{on}\quad\partial \Omega,
\end{equation}
where $\lambda=\lambda_\rho,\;\Lambda= \Lambda_\rho$ are the first eigenvalues or the basic frequencies and $u=u(x), \; v=v(x)$ are the corresponding eigenfunctions or the lateral displacements. The operator $\Delta^2$ stands for usual bi-Laplacian, that is $\Delta^2 u=\Delta(\Delta u)$.  The principal eigenvalue $\lambda$ of problem
\eqref{mpde1} is obtained by minimizing the associate Rayleigh quotient

\begin{equation}  \label{ray1}
\lambda= \inf \{\frac{\int_{\Om} (\Delta w)^2 dx}{\int_{\Om}  \rho w^2 dx}: w \in H^2(\Om)\cap H^1_0(\Om), w\neq 0\},
\end{equation}
and the  first eigenvalue $\Lambda$ of problem \eqref{mpde2} is obtained by minimizing the associate Rayleigh quotient
\begin{equation}  \label{ray2}
\Lambda= \inf \{\frac{\int_{\Om} (\Delta w)^2 dx}{\int_{\Om}  \rho w^2 dx}: w \in H^2_0(\Om), w\neq 0\},
\end{equation}
where it is well known \cite{stru} that the inferior is attained in both cases. By regularity results the solutions to problems \eqref{mpde1} and \eqref{mpde2} belongs
to $H^4_{loc}(\Om)$ and these equations hold a.e. in $\Om$ \cite{agmon}.

 To determine the system's profile which gives  the maximum and minimum  principal eigenvalues, Cuccu \emph{et al}  have verified the following optimization problems
 \begin{equation}  \label{maxp1}
\max_{\rho \in  \mathcal{P} }\lambda_{\rho},
\end{equation}
\begin{equation}  \label{minp1}
\min_{\rho \in  \mathcal{P} }\lambda_{\rho},
\end{equation}
\begin{equation}  \label{maxp2}
\max_{\rho \in  \mathcal{P} }\Lambda_{\rho},
\end{equation}
\begin{equation}  \label{minp2}
\min_{\rho \in  \mathcal{P} }\Lambda_{\rho},
\end{equation}
in \cite{cuccu, cuccu2,anedda}. The existence of solutions for problems \eqref{maxp1} - \eqref{minp1} and  \eqref{minp2} have been proved for general domain $\Om$.  But, the existence of a solution for problem \eqref{maxp2} has been established when $\Om$ is a positivity preserving domain for $\Delta ^2 u$
under homogeneous Dirichlet boundary conditions. For instance, the ball is a domain that enjoys such a property. In spite of these existence results, the precise identifications of the maximums and minimums were found only in case $\Om$ is a ball.

 In eigenvalue optimization for elliptic partial differential equations, one of challenging mathematical problems after the problem of existence is an exact formula of the optimizer or optimal shape design. Most papers in this field answered this question just in case $\Om$ is a ball. For other domains qualitative properties of solutions were investigated and partial answers were given \cite{chanillo,conca,cuccu2011, derlet, kurata, lurian, abbasali,abbasali2}.
 From the physical point of view, it is important  to know the shape of the optimal density functions in case $\Om$ is not a ball.

  This class of problems is difficult to solve because of the lack of the topology information of the optimal shape. There must be
  numerical approaches to determine the optimal shape design. The mostly used  methods now are the homogenization method \cite{alliare} and the level set
  method \cite{osher}. The level set method is well known for its capability  to handle topological  changes, such as breaking one component into several, merging several components into one and forming sharp corners. This approach has been applied to the study of extremum problems of eigenvalues of inhomogeneous
  structures including the identification of composite membranes with extremum eigenvalues \cite{oshersantosa,hekao,zhu}, design of composite  materials with a desired spectral gap or maximal spectral gap \cite{kaoyab},
  finding optical devices that have a high quality factor \cite{kaosantosa2} and principle eigenvalue optimization in population biology \cite{kaoluo}.

   Recently, Kao and Su \cite{kao2} proposed an efficient rearrangement algorithm based on the Rayleigh quotient formulation of eigenvalues. They have solved minimization and maximization problem for the $k$-th eigenvalue ($k\geq 1$) and maximization of spectrum ratios of the second order elliptic differential operator in $\mathbb{R}^2$. We extend the approach to solve a fourth order partial differential equation. Most of the previous results are for second order
operators with Dirichlet boundary conditions and the geometric constraint have been considered for $m=2$. Here we study a partial differential equation with different boundary conditions  and we develop our algorithms for general $m\geq 2$. It is common in the literature that $\Om$ have been considered as a rectangle \cite{kong}.
To show the capacity and efficiency of the numerical method, we apply it to solve problems \eqref{maxp1}-\eqref{minp2} when $\Om$ has a more complicated geometrical structure than a rectangle. The algorithms  start from a given density $\rho_0$ and converge to the optimizers.

  Throughout this paper we shall write increasing instead  of non- decreasing, and decreasing instead of non- increasing.

 This section is closed with some definitions from  the rearrangement theory  related to our optimization problems. The reader can refer to  \cite{bu89,Alvino} 
for further information about rearrangement theory.
\begin{defn}\label{readef}
Two Lebesgue measurable functions  $\rho: \Om \rightarrow \Bbb{R}$, $\rho_0:\Om \rightarrow \Bbb{R}$, are said to be rearrangements of each other if\\
 \begin{equation}\label{rea}
|\{x\in \Om : \rho(x)\geq r \}|=|\{x\in \Om : \rho_0(x)\geq r\}|\qquad~\quad\forall r \in \mathbb{R}.
\end{equation}
\end{defn}
The notation $\rho\sim \rho_0$ means that $\rho$ and $\rho_0$ are rearrangements of each other. Consider $\rho_0:\Om \rightarrow \Bbb{R}$, the class of rearrangements generated by $\rho_{0}$, denoted $\mathcal{P}$, is defined as follows\\
 \begin{equation*}
\mathcal{P}=\{\rho:\rho\sim \rho_{0}\}.
\end{equation*}
\section{Numerical Algorithms}
 In this section we describe  the algorithms that  start from a given initial density function and converge  to the optimizers. Consider the case that $m$ different materials with densities $0<c_1<c_2<...<c_m$ are distributed in arbitrary
 pairwise disjoint measurable subsets $D_i$, $i=1..m$, respectively of $\Omega$ so that $\cup_1^m D_i= \Om$ and $|D_i|=S_i$. Then, set  $ \rho_0(x)=c_1\chi_{D_1}+...+c_m \chi_{D_m}$ and we have the following technical assertion.
 \begin{lem}\label{chirho}
  Function $\rho$ belongs to the rearrangement class $\mathcal{P}$ if and only if  $\rho=c_1\chi_{D^\prime_1}+...+c_m \chi_{D^\prime_m}$ where $D^\prime_i$, $i=1..m$, are  pairwise disjoint subsets of $\Om$ such that $|D^\prime_i|=S_i$ and $\cup_1^m D^\prime_i= \Om$.
 \end{lem}
 \begin{proof}
  Assume $\rho \in \mathcal{P}$. In view of definition \ref{readef},
  \begin{align*}
|\{x\in \Omega: \rho_0(x)= r\}|=|\cap_1 ^\infty \{x\in \Omega: r \leq \rho_0(x)< r +\dfrac{1}{n}\}|\\
=\lim_{n\rightarrow \infty}
|\{x\in \Omega:  \rho_0(x)\geq r\}|- |\{x\in \Omega:  \rho_0(x) \geq r +\dfrac{1}{n}\}|\\
=\lim_{n\rightarrow \infty}
|\{x\in \Omega:  \rho(x)\geq r\}|- |\{x\in \Omega:  \rho(x) \geq r +\dfrac{1}{n}\}|\\
=|\cap_1 ^\infty \{x\in \Omega: r \leq \rho(x)< r +\dfrac{1}{n}\}|= |\{x\in \Omega: \rho(x)= r\}|,
\end{align*}
  where it means that the
  level sets of $\rho$ and $\rho_0$  have the same measures and this yields the assertion. The other part of the theorem is concluded from
  definition \ref{readef}.
 \end{proof}

 Now, we propose  algorithms  to derive the solutions of problems \eqref{maxp1}-\eqref{minp2} respectively. For the maximization problems, we start from a given
initial density functions $\rho_0 $ and  extract new  density functions  $\rho_1$ using the eigenfunction of equations \eqref{mpde1}- \eqref{mpde2} such that the first eigenvalues are
increased, i.e.
\begin{equation*}
\lambda_{\rho_0}\leq\lambda_{\rho_1},\quad \Lambda_{\rho_0}\leq\Lambda_{\rho_1}.
\end{equation*}
In the same spirit as in the maximization cases, we initiate from   a given density functions $\rho_0$ and extract another density functions $\rho_1  $ such that principal eigenvalues of \eqref{mpde1}- \eqref{mpde2} are decreased, i.e.
\begin{equation*}
\lambda_{\rho_0}\geq\lambda_{\rho_1},\quad \Lambda_{\rho_0}\geq\Lambda_{\rho_1}.
\end{equation*}

 The algorithms strongly based on the Rayliegh quotients
in formulas \eqref{ray1}-\eqref{ray2}. Such  algorithms  have
been applied successfully to minimize eigenvalues of some second order elliptic operators \cite{chanillo,kao1,kao2}. They rely on the variational formulation of
the eigenvalues and use level sets of the eigenfunctions or gradient
of them.  Employing the
level sets of the eigenfunctions, we need some
results of the rearrangement theory with an eye on our problem \cite{bu89}.

\begin{lem}\label{ber}
Let $\mathcal{P}$ be the set of rearrangements of a fixed function
$\rho_{0}\in L^r(\Omega)$, $r>1$, $\rho_{0}\not\equiv 0$, and let $q\in
L^s(\Omega)$, $s=r/(r-1)$, $q\not\equiv 0$. If there is an
increasing function $\xi:\mathbb{R}\rightarrow \mathbb{R}$  such that $\xi(q)\in \mathcal{P}$, then
\begin{equation*}
\int_{\Omega} \rho q dx \leq \int_{\Omega} \xi(q)q dx~~~~~~~
\qquad\qquad\forall ~\rho \in\mathcal{P},
\end{equation*}
and the function $\xi(q)$ is the unique maximizer relative to
$\mathcal{P}$. Furthermore, if there is a decreasing
function $\eta:\mathbb{R}\rightarrow \mathbb{R}$ such that $\eta(q)\in \mathcal{P}$, then
\begin{equation*}
\int_{\Omega} \rho q dx \geq \int_{\Omega} \eta(q)q
dx~~~~~~~\qquad\qquad \forall ~\rho \in\mathcal{P},
\end{equation*}
and the function $\eta(q)$ is the unique minimizer relative to
$\mathcal{P}$.
\end{lem}

 Next two lemmas provide our main tool for constructing the numerical algorithms.
\begin{lem}\label{bathtub1}
Let $f(x)$ be a nonnegative function in $L^1(\Om)$ such that its level sets have measure zero. Then the minimization problem
\begin{equation}\label{bathtubinf}
\inf_{\rho \in \mathcal{P}} \int_{\Om}\rho f dx,
\end{equation}
is uniquely solvable by $\widehat{\rho}(x)=c_1\chi_{\widehat{D}_1}+...+c_m\chi_{\widehat{D}_m}$ where $|\widehat{D}_i|=S_i$, $i=1..m$, and
\begin{equation*}
\widehat{D}_1=\{x\in \Om:\,\:f(x)\geq t_1\},
\end{equation*}
\begin{align*}
t_1=\sup\{s\in \mathbb{R} : |\{x\in \Om:\,\:f(x)\geq s\}|\geq S_1\},
\end{align*}
\begin{equation*}
\widehat{D}_i=\{x\in \Om\backslash \cup _{j=1}^{i-1} \widehat{D_j}:\,\:f(x)\geq t_i\},
\end{equation*}
\begin{align*}
t_i=\sup\{s\in \mathbb{R} : |\{x\in \Om\backslash \cup _{j=1}^{i-1} \widehat{D_j}:\,\:f(x)\geq s\}|\geq S_i\},
\end{align*}
for $i=2..m-1$.
\end{lem}
\begin{proof}
Let us rearrange the super level sets of the function $f(x)$ with an eye on the sets $\widehat{D_i}$, $i=1..m$. We should mention here that the pairwise disjoint sets $\widehat{D_i}$, $i=1..m$ are determined uniquely since the level sets of the function $f(x)$ have  measure zero. In addition,
\begin{equation*}
t_1>t_2>...>t_{m-1},
\end{equation*}
since we have $S_i>0$,  $i=1..m$. This motivates the following decreasing function
\begin{equation*}
 \eta(t)=
 \left\{
     \begin{array}{ll}
      c_1\quad\;\;\;\quad t\geq t_1,
       \\ c_2 \quad \quad  \;\; \;     t_2\leq  t < t_1,
       \\ \vdots
       \\ c_{m-1} \quad \quad       t_{m-1}\leq  t < t_{m-2},
       \\ c_m \quad  \quad \quad otherwise,
     \end{array}
   \right.
 \end{equation*}
 where it yields
 \begin{equation*}
 \eta(f(x))=c_1\chi_{\widehat{D}_1}+...+c_m\chi_{\widehat{D}_m}.
  \end{equation*}
 Invoking lemma \ref{chirho}, this function belongs to the rearrangement class $\mathcal{P}$. According to lemma \ref{ber}, $ \eta(f(x))$ is the unique minimizer of the problem \eqref{bathtubinf}.
\end{proof}

\begin{lem}\label{bathtub2}
Let $f(x)$ be a nonnegative function in $L^1(\Om)$ such that its level sets have measure zero. Then the maximization problem
\begin{equation}\label{bathtubsup}
\sup_{\rho \in \mathcal{P}} \int_{\Om}\rho f dx,
\end{equation}
is uniquely solvable by $\widehat{\rho}_(x)=c_1\chi_{\widehat{D}_1}+...+c_m\chi_{\widehat{D}_m}$ where $|\widehat{D}_i|=S_i$, $i=1..m$, and
\begin{equation*}
\widehat{D}_1=\{x\in \Om:\,\:f(x)\leq t_1\},
\end{equation*}
\begin{align*}
t_1=\inf\{s\in \mathbb{R} : |\{x\in \Om:\,\:f(x)\leq s\}|\geq S_1\},
\end{align*}
\begin{equation*}
\widehat{D}_i=\{x\in \Om\backslash \cup _{j=1}^{i-1} \widehat{D_j}:\,\:f(x)\leq t_i\},
\end{equation*}
\begin{align*}
t_i=\inf\{s\in \mathbb{R} : |\{x\in \Om\backslash \cup _{j=1}^{i-1} \widehat{D_j}:\,\:f(x)\leq s\}|\geq S_i\},
\end{align*}
for $i=2..m-1$.
\end{lem}
\begin{proof}
The proof is  similar to that of lemma \ref{bathtub1} and is omitted.
\end{proof}

Lemmas \ref{bathtub1} and \ref{bathtub2} allow us to derive  sequences of density functions  $\rho_n$
such that corresponding eigenvalues $\la(\rho_n)$ and $\Lambda(\rho_n)$
 are monotone  sequences of eigenvalues. To do this, we need the following theorem.

\begin{thm}\label{decseq}
 Assume $ \rho_0(x)=c_1\chi_{D_1}+...+c_m \chi_{D_m}$ is a member of $\mathcal{P}$. Then, there exist two functions $ \rho_1$ and $\rho^\prime_1$ in the rearrangement class $\mathcal{P}$
 such that
\begin{equation*}
\la( \rho_0)\geq \la( \rho_1), \quad \Lambda( \rho_0)\geq \Lambda( \rho^\prime_1).
\end{equation*}
\end{thm}
\begin{proof}
We prove the assertion for equation \eqref{mpde1}. For  \eqref{mpde2}, the proof is similar to that of equation  \eqref{mpde1} and it is omitted.
Inserting $\rho_0$ as the density function
 into \eqref{mpde1}, we find $u_0$ as the eigenfunction of the
 equation corresponding to $\la_0=\la(\rho_0)$.

   We claim that the level sets of  $u_0$  have  measure zero.
As we mentioned, $u_0$ satisfies  equation \eqref{mpde1}  a.e. in $\Om$. This means
 \begin{equation*}
 \Delta^2 u_0= \lambda \rho u_0 ,\quad \mathrm{a.e.\:\:\:in}\quad\Omega.
 \end{equation*}
 If the set $E=\{x\in \Om:\quad u_0(x)=0\}$ has  positive measure then $u_0$ is identically zero by theorem 4.4 of \cite{cuccu2}.This yields that
 $\Delta^2 u_0\neq 0$ a.e. in $\Om$. Consequently,  the level sets of  $u_0$  have  measure zero applying lemma 7.7 of \cite{gilbarg}.

 Set $f(x)= u_0^2(x)$ in lemma \ref{bathtub2}, then one can achieve $\rho_1$ in  $\mathcal{P}$ such that
 \begin{equation*}
\int_{\Om} \rho_0  u_0^2dx\leq
\int_{\Om} \rho_1   u_0^2dx.
 \end{equation*}
Therefore,
\begin{equation*}
\frac{\int_{\Om}(\Delta u_0)^2dx}{\int_{\Om} \rho_0 u_0^2dx}
\geq \frac{\int_{\Om}(\Delta u_0)^2dx}{\int_{\Om} \rho_1  u_0^2dx},
 \end{equation*}
 then
 \begin{equation*}
\la( \rho_0)\geq \la( \rho_1),
\end{equation*}
based on \eqref{ray1}.
\end{proof}

 Utilizing  theorem \ref{decseq}, we can derive two decreasing sequences of eigenvalues
\begin{equation*}
\la(\rho_{n-1})\geq \la(\rho_{n}), \quad \Lambda(\rho_{n-1})\geq \Lambda(\rho_{n}).
\end{equation*}
Obviously, these  decreasing sequences are bounded   below by zero and  so they converge.
In view of theorem \ref{decseq}, we can propose an iterative procedure to find the minimal density configurations for the
eigenvalue minimization problems \eqref{minp1} and   \eqref{minp2}.

 The maximization problems are more complicated and  an iterative method cannot be derived by arguments similar to those in the minimization cases \cite{kao2}. If we consider $\rho_0$ as an arbitrary  function in  $\mathcal{P}$ and $u_0$ as an associated eigenfunction of  \eqref{mpde1},  then
 one can find a density function $\rho_1$ in $\mathcal{P}$ regarding lemma \ref{bathtub1} such that
 \begin{equation*}
\int_{\Om} \rho_0  u_0^2dx\geq
\int_{\Om} \rho_1   u_0^2dx,
 \end{equation*}
and then
\begin{equation*}
\la(\rho_0)\leq\frac{\int_{\Om}(\Delta u_0)^2dx}{\int_{\Om} \rho_0 u_0^2dx}
\leq \frac{\int_{\Om}(\Delta u_0)^2dx}{\int_{\Om} \rho_1  u_0^2dx}\geq \la(\rho_1).
 \end{equation*}
The same argumentation is valid for the principal eigenvalues of \eqref{mpde2} as well. Hence, we cannot produce an increasing sequence of eigenvalues since the next generated eigenvalue may be less than the previous one. Inspired by the method introduced in  \cite{kao2},  the strategy to guarantee a monotone increasing sequence is to add an acceptance rejection method. If this new eigenvalue increases the eigenvalue, the density function will be accepted. Otherwise, the partial swapping method will be used. Indeed, in the partial swapping method we use $\rho_1$ where $\delta \rho =\rho_1-\rho_0$ is small enough.
\begin{thm}\label{psmt}
Let $\rho_0,\:\rho_1$ be functions in $\mathcal{P}$ where
\begin{equation}\label{smallrho}
\int_{\Om} \rho_0  u_0^2dx>
\int_{\Om} \rho_1   u_0^2dx,
\end{equation}
and $\|\delta \rho\|_{\lt}$ is small enough. Then,
\begin{equation*}
\la(\rho_{0})<\la(\rho_{1}), \quad \Lambda(\rho_{0})< \Lambda(\rho_{1}).
\end{equation*}
\end{thm}
\begin{proof}
We prove the assertion for equation \eqref{mpde1}. For  \eqref{mpde2}, the proof is similar to that of equation  \eqref{mpde1} and is omitted. Assume $u_0$ and $u_1$ are eigenfunctions of \eqref{mpde1} corresponding to $\rho_0$ and $\rho_1$ respectively where  normalized so that $\|\Delta u_0\|_{\lt}=\|\Delta u_1\|_{\lt}=1$. Applying \eqref{ray1}, we have
\begin{equation*}
\la(\rho_{0})=\frac{1}{\int_{\Om} \rho_0  u_0^2dx}\leq \frac{1}{\int_{\Om} \rho_0  u_1^2dx},
\end{equation*}
which it yields
\begin{equation}\label{thm2-1}
\int_{\Om} \rho_0  u_1^2dx \leq \int_{\Om} \rho_0  u_0^2dx.
\end{equation}
Then,
\begin{align*}
\frac{1}{\la(\rho_{1})}-\frac{1}{\la(\rho_{0})}=\int_{\Om} \rho_1  u_1^2dx-\int_{\Om} \rho_0  u_0^2dx\\
=\int_{\Om} \rho_1  u_1^2- \rho_1  u_0^2+ \rho_1  u_0^2- \rho_0  u_0^2dx\\
=\int_{\Om} (\rho_0+\delta \rho)(u_1^2-u_0^2) dx+ \int_{\Om} \delta \rho u_0^2 dx
\\=\int_{\Om} \rho_0(u_1^2-u_0^2) dx+\int_{\Om} \delta\rho(u_1^2-u_0^2) dx+\int_{\Om} \delta \rho u_0^2 dx.
\end{align*}
On the right hand side of the  last equality, we have three integrals where the first and the last one from the left are negative according to
\eqref{smallrho} and \eqref{thm2-1}. We claim that $\|u_1-u_0\|_{\lt}\rightarrow 0$ as $\|\delta \rho\|_{\lt}\rightarrow 0$. Then, one can observe that
the second integral converges to zero with the rate of convergence $O(\|\delta \rho\|^{1+s}_{\lt})$ for some  $s>0$  and the third integral converges to zero with the rate of convergence $O(\|\delta \rho\|_{\lt})$. Hence if  $\|\delta \rho\|_{\lt}$ is small enough, we can infer that the right hand side of the last equality is negative and
\begin{equation*}
\la(\rho_{0})<\la(\rho_{1}).
\end{equation*}

 It remains to prove the claim. Applying the dominate convergence theorem, we deduce  that if $\rho_0 \rightarrow  \rho_1$ strongly in $\lt$ then
 $\rho_0 \rightharpoonup  \rho_1$ weakly in $L^\infty(\Om)$. By the same reasoning, used in the proof of lemma 3.1 in \cite{cuccu}, it can be concluded that
 $u_0\rightharpoonup u_1$ weakly in $H^2(\Om)$ and $u_0\rightarrow u_1$ strongly in $\lt$.
\end{proof}
\begin{rem}\label{boundedseqpsm}
Let $\rho_0(x)=c_1\chi_{D_1}+...+c_m \chi_{D_m}$ be a function in $\mathcal{P}$. We derive  $\rho_1$ stated in \eqref{smallrho} in the following form
\begin{equation*}
\rho_1(x)=c_1\chi_{D_1}+...+ c_j \chi_{D^\prime_j}+...+ c_i \chi_{D^\prime_i}+...+c_m \chi_{D_m},
\end{equation*}
such that $D^\prime_j= (D_j-A)\cup B$ and $D^\prime_i= (D_i-B)\cup A$ where  $A$ is a subset of $D_j$ and $B$ is a subset of $D_i$
so that $|A|=|B|$. Then,
\begin{equation*}
\delta \rho=\rho_1-\rho_0=c_j(\chi_{B}-\chi_{A})+c_i(\chi_{A}-\chi_{B}),
\end{equation*}
and functions $\rho_1$, $\rho_0$ satisfies \eqref{smallrho} if
\begin{equation*}
\int_B u^2_0> \int_A u^2_0.
\end{equation*}
It is noteworthy  that    $\|\delta \rho\|_{\lt}$ will be small enough if one adjusts $|A|=|B|$ small enough. The sets $A$ and $B$ are selected by trail and error.
\end{rem}

\begin{rem}\label{boundedseq}
Utilizing lemma \eqref{bathtub1} and theorem \ref{psmt}, we can derive an increasing sequences of eigenvalues
\begin{equation*}
\la(\rho_{n-1})\leq \la(\rho_{n}), \quad \Lambda(\rho_{n-1})\leq \Lambda(\rho_{n}).
\end{equation*}
Using variational formulation \eqref{ray1} and \eqref{ray2}, it can be said that these sequences are bounded  above. Consider an entire density function $\rho$ in $\mathcal{P}$. Then from equation \eqref{ray1} or \eqref{ray2} and lemma \ref{bathtub1} we  conclude
\begin{equation*}
\la_{\rho}, \Lambda_{\rho}\leq \frac{\int_{\Om}(\Delta \psi)^2dx}{\int_{\Om}\rho \psi^2dx}\leq \frac{\int_{\Om}(\Delta \psi)^2dx}{\int_{\Om}\overline{\rho} \psi^2dx},
\end{equation*}
where $\psi$ is the eigenfunction associated with the principal
eigenvalue of the Laplacian with Dirichlet boundary conditions and  $\overline{\rho}$ is the minimizer stated in lemma \ref{bathtub1} for $f=\psi^2$. Consequently, the increasing sequences of eigenvalues are bounded above and are  convergent.
 \end{rem}

 \begin{table}[h]
\caption{}
\centering 
\begin{tabular}{ l} 
\hline 
\hline
\textbf{Algorithm $2$.} Eigenvalue minimization \\ [1 ex] 
\hline 
\hline
\textbf{Data:} An initial density function $\rho_0$  \\ 
\textbf{Result:} A sequence of decreasing eigenvalues $\la(\rho_n)$  \\
\textbf{$1$.} Set $n = 0$;  \\
\textbf{$2$.} Compute $u_{n}$ and $\la(\rho_n)$; \\
\textbf{$3$.} Compute $\rho_{n+1}$ applying lemma \ref{bathtub2}; \\
\textbf{$4$.} If $ \|\delta \rho\|_{\lt}< TOL$ then stop;\\
$\quad$ else\\
$\qquad$ $\qquad$  Set $n=n+1$;\\
 $\qquad$ $\qquad$ Go to step $2$;\\
 [1ex] 
\hline 
\end{tabular}
\label{table1} 
\end{table}
\section{Implementation of the numerical algorithms}\label{num}

 This section provides us with the details of the implementation for
algorithms introduced in the previous section  and some
examples are chosen to illustrate the numerical solutions. The algorithms work for both equations \eqref{mpde1}  and \eqref{mpde2} similarly and so we
state the procedures just for  \eqref{mpde1}.

 At iteration step $n$, there is a guess
for the configuration of the optimal density function where it  is denoted by $\rho_n$. We use the finite
element method with  piecewise linear basis
functions to discretize equation \eqref{mpde1} with $\rho_n$.
Let $u_n$ be an
eigenfunction of \eqref{mpde1}  associated with
eigenvalue $\la_n= \la(\rho_n)$. For  minimization problem \eqref{minp1},
 we should extract a new density function $\rho_{n+1}$  based upon the level sets of eigenfunction $u_n$ where it belongs to $\mathcal{P}$
and $\la(\rho_{n})> \la(\rho_{n+1})$. To derive this $\rho_{n+1}$,
 we make use of  lemma  \ref{bathtub2} and identify $\rho_{n+1}$ by setting $f(x)=u_n^2(x)$.
  According to theorem \ref{decseq}, we have $\la(\rho_{n})> \la(\rho_{n+1})$ and the
       generated sequence is convergent. The resulting algorithm is shown in table 1. There is a stopping criterion in this method.
        The algorithm  stops
         when $\|\delta \rho\|_{\lt}=\| \rho_{n+1}-\rho_{n}\|_{\lt}$ is less than a prescribed tolerance $TOL$.

\begin{table}[h]
\caption{}
\centering 
\begin{tabular}{ l} 
\hline 
\hline
\textbf{Algorithm $1$.} Eigenvalue maximization \\ [1 ex] 
\hline 
\hline
\textbf{Data:} An initial density function $\rho_0$  \\ 
\textbf{Result:} A sequence of increasing eigenvalues $\la(\rho_n)$  \\
\textbf{$1$.} Set $n = 0$;  \\
\textbf{$2$.} Compute $u_{n}$ and $\la(\rho_n)$; \\
\textbf{$3$.} Compute $\rho_{n+1}$ applying lemma \ref{bathtub1}; \\
\textbf{$4$.} Compute $\la (\rho_{n+1})$;\\
\textbf{$5$.} If $\la (\rho_{n})< \la (\rho_{n+1})$ then go to step $6$;\\
$\quad$ else\\
$\qquad$ $\qquad$  Compute $\rho_{n+1}$ applying  remark \ref{boundedseqpsm};\\
\textbf{$6$.} If $\|\delta \rho\|_{\lt}| < TOL$ then stop;\\
$\quad$ else\\
$\qquad$ $\qquad$  Set $n=n+1$; Go to step $2$;\\
 [1ex] 
\hline 
\end{tabular}
\label{table1} 
\end{table}

For  maximization problem \eqref{maxp1},
 we should extract a new density function $\rho_{n+1}$   where it belongs to $\mathcal{P}$
and $\la(\rho_{n})< \la(\rho_{n+1})$. To derive this $\rho_{n+1}$,
 we make use of  lemma  \ref{bathtub1} and identify $\rho_{n+1}$ by setting $f(x)=u_n^2(x)$. If $\la(\rho_{n})< \la(\rho_{n+1})$, the derived  function
$\rho_{n+1}$  is accepted. Otherwise, the partial swapping method introduced in theorem \ref{psmt} and remark \ref{boundedseqpsm} is used to generate the new $\rho_{n+1}$. The resulting algorithm is shown in table 2. According to remark \ref{boundedseq}, this increasing sequence is convergent.

In the above algorithms we need to calculate parameters $t_1..t_{m-1}$ which is mentioned in lemma \ref{bathtub1} or \ref{bathtub2}. In order to find $t_1..t_{m-1}$, two algorithms are developed which they apply the idea of the bisection method. Both algorithms are the same in essence and we only state the algorithm related to the maximization problem.  Introducing the distribution function $F(s)=|\{x\in \Om: \quad u^2_n(x)> s \}|$, we state algorithm 3  in table 3 to compute $t_1..t_{m-1}$. Again, we should consider a tolerance $TOL$ in this algorithm since it is meaningless computationally to find a set $D_k$ satisfying $|D_k|=S_k$ exactly.

\begin{table}[h]
\caption{}
\centering 
\begin{tabular}{ l} 
\hline 
\hline
\textbf{Algorithm $3$.} Bisection method for $t_1..t_{m-1}$ \\ [1 ex] 
\hline 
\hline
\textbf{Data:} Eigenfunction $u_n$ on the domain $\Om$ \\ 
\textbf{Result:} The level $t_1..t_{m-1}$  \\
\textbf{For $i=1..m-1$ do} \\
\textbf{$1$.} Set $L= 0, \quad U=\underset{x\in\Om}{\max}\,\, u^2_n(x)$  \\
\textbf{$2$.}  Set  $\theta=(L+U)/2$; \\
\textbf{$3$.} If $|F(\theta)-S_i|< TOL$ then set $t_i=\theta$ and  $\Om=\Om\backslash D_i$; \\
$\quad$else \\
\,\,\,\,\,$\quad$ If $F(\theta)<S_i$ then \\
\,\,\,\,\,$\quad$ $\quad$ Set $U=\theta$; Go to step 2;\\
\,\,\,\,\,$\quad$ else\\
\,\,\,\,\,$\quad$ $\quad$ Set $L=\theta$; Go to step 2;\\
 [1ex] 
\hline 
\end{tabular}
\label{table1} 
\end{table}

  \begin{figure}[h]\label{figmax-1}
  \centering
  \subfloat[$\lambda_{\max}=1.51$]{\label{figmax-1:r}\includegraphics[width=0.2\textwidth]{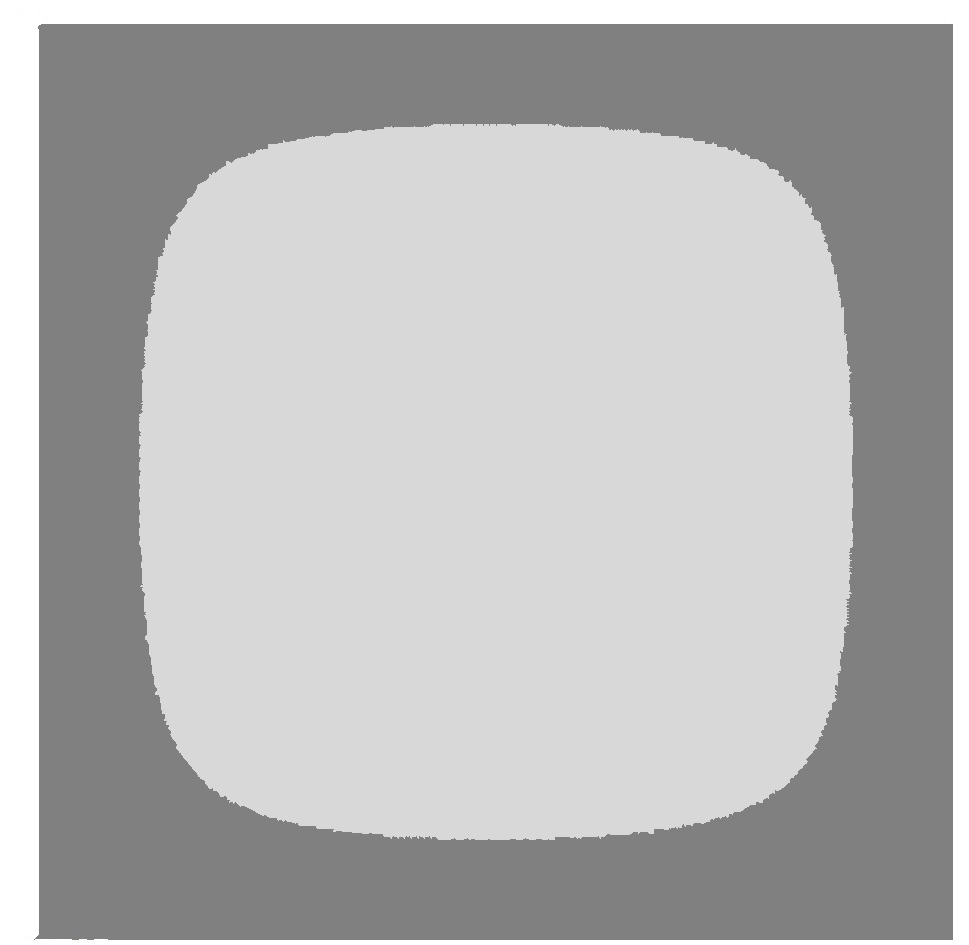}}
  \subfloat[$\lambda_{\max}=6.28$]{\label{figmax-1:c}\includegraphics[width=0.22\textwidth]{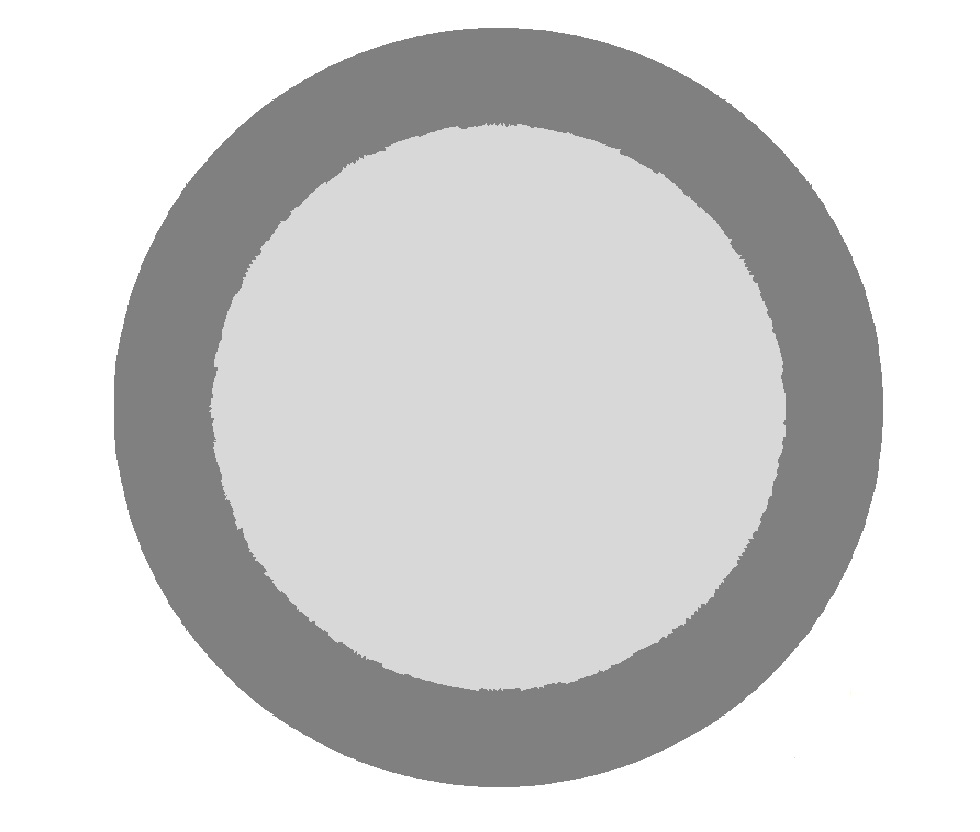}}
    \caption{The maximizer sets in black}
  \label{figmax-1}
\end{figure}

\begin{figure}[h]\label{figmax-2}
  \centering
  \subfloat[$\lambda_{\max}=4.51$]{\label{figmax-1:e}\includegraphics[width=0.23\textwidth]{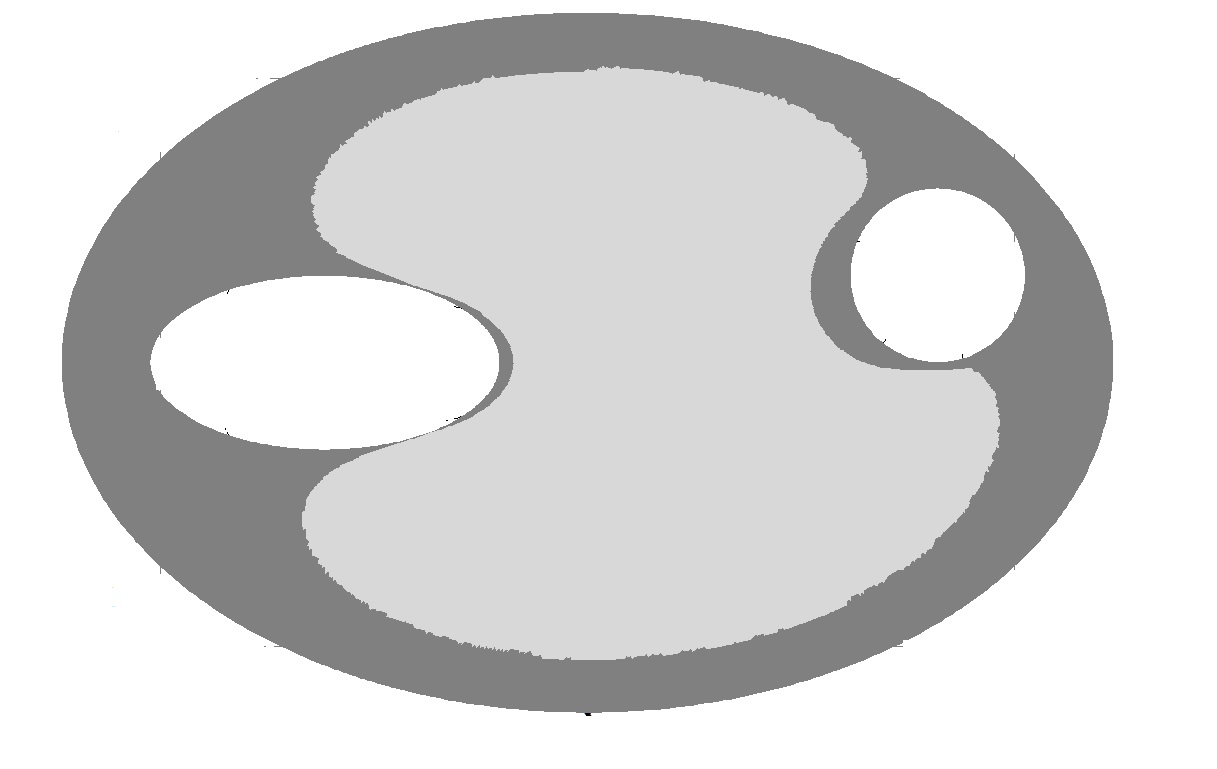}}
  \subfloat[$\lambda_{\max}=146.23$]{\label{figmax-1:cr}\includegraphics[width=0.2\textwidth]{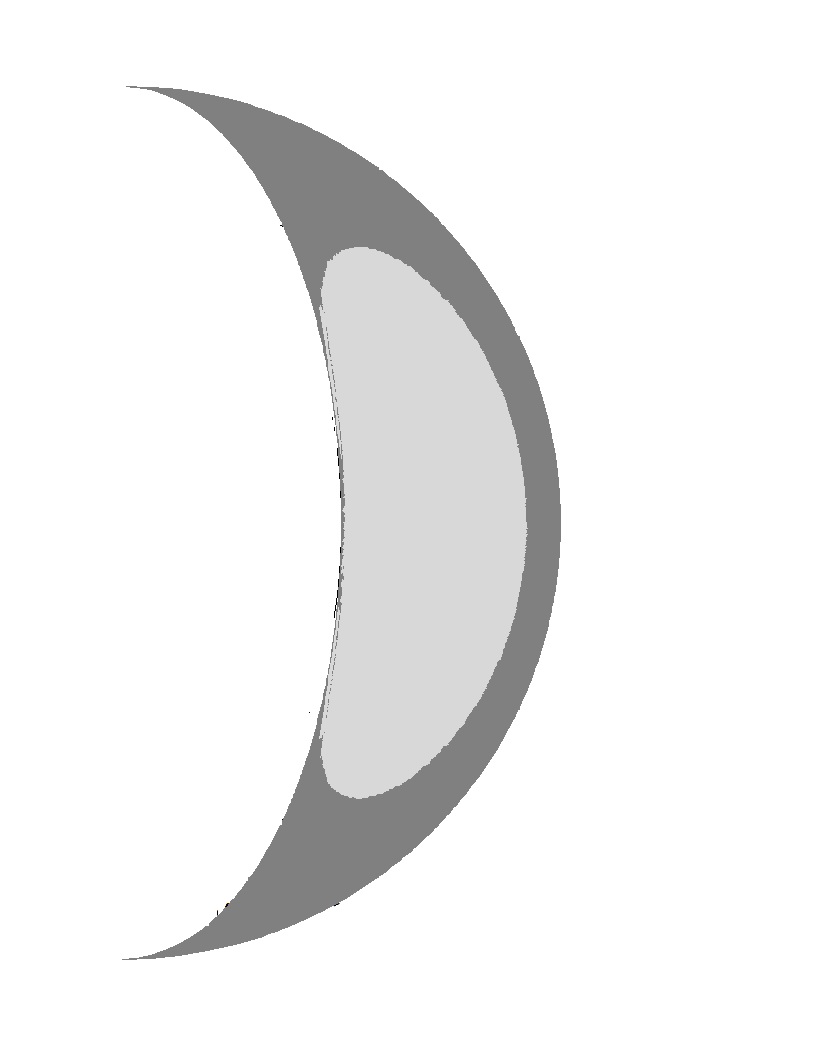}}
    \caption{The maximizer sets in black}
  \label{figmax-2}
\end{figure}

 \begin{figure}[t]\label{figmin-1}
  \centering
  \subfloat[$\lambda_{\min}=0.90 $]{\label{figmin-1:r}\includegraphics[width=0.2\textwidth]{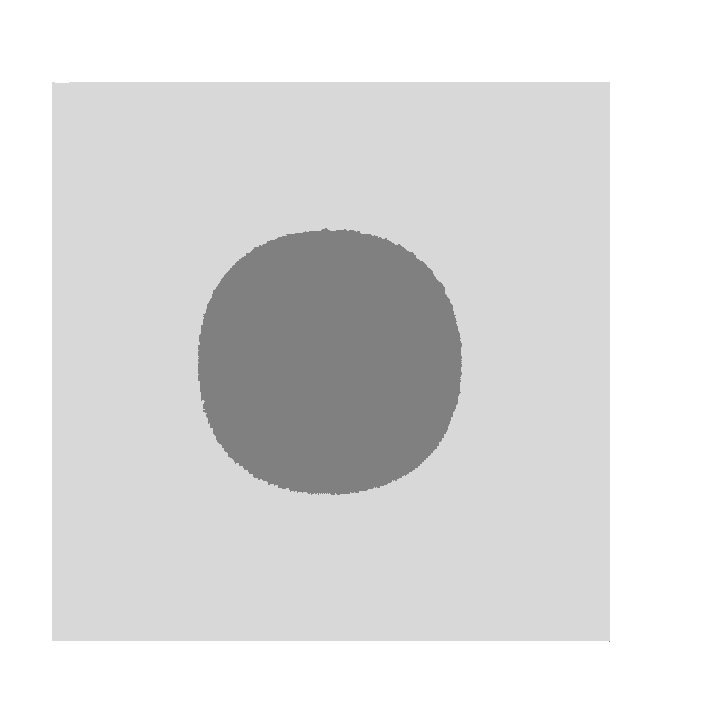}}
  \subfloat[$\lambda_{\min}=1.24 $]{\label{figmin-1:c}\includegraphics[width=0.2\textwidth]{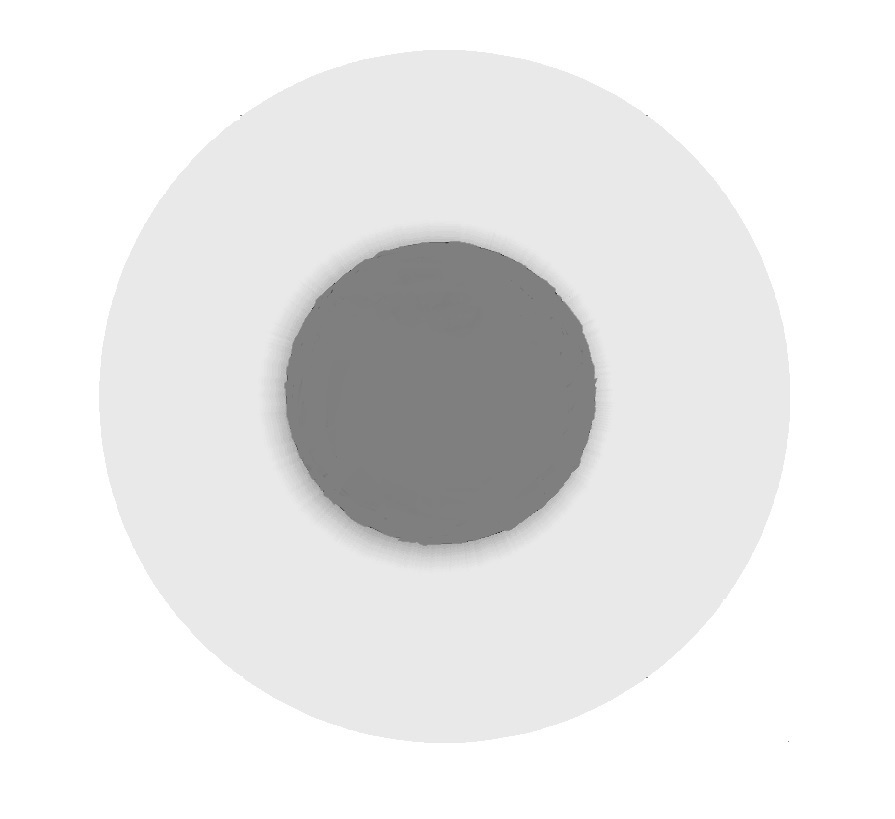}}
    \caption{The minimizer sets in black}
  \label{figmin-1}
\end{figure}

\begin{figure}[t]\label{figmin-2}
  \centering
  \subfloat[$\lambda_{\min}=2.58$]{\label{figmin-2:e}\includegraphics[width=0.23\textwidth]{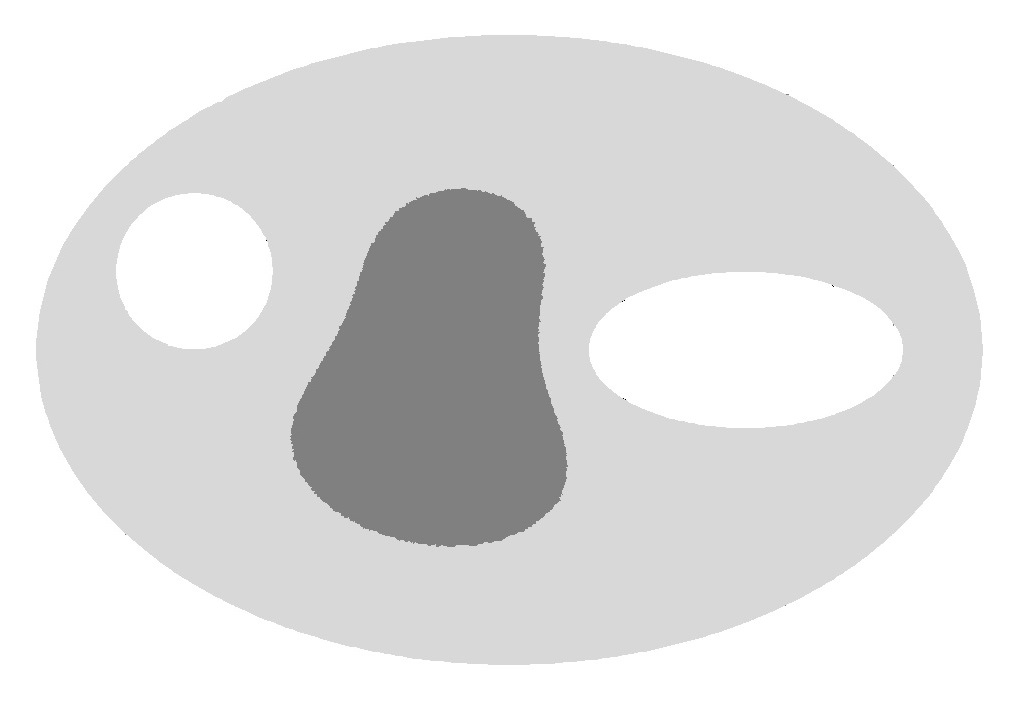}}
  \subfloat[$\lambda_{\min}=81.11$]{\label{figmin-2:cr}\includegraphics[width=0.2\textwidth]{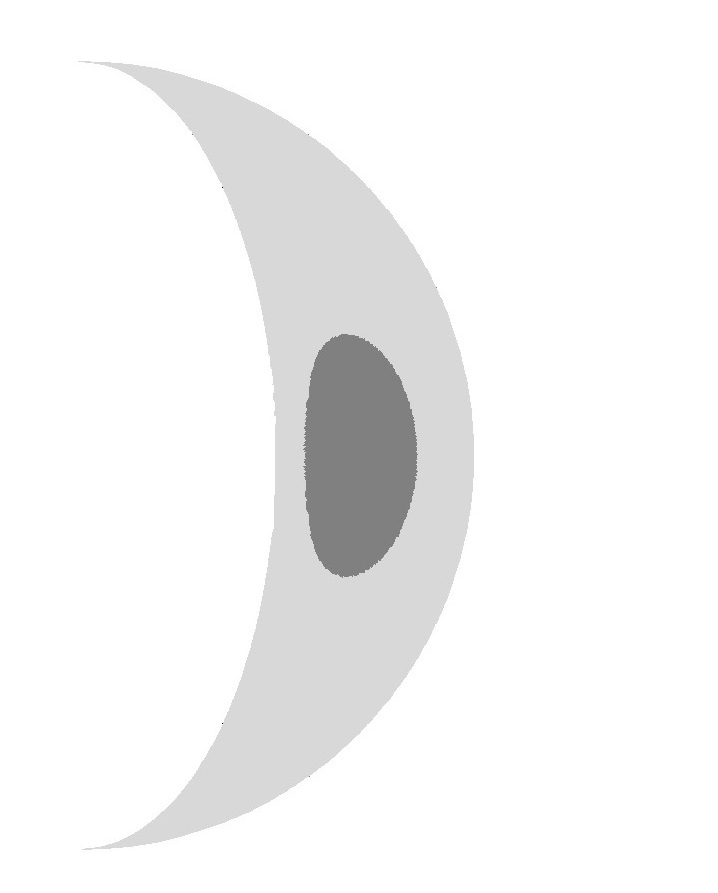}}
   \caption{The minimizer sets in black}
  \label{figmin-2}
\end{figure}

  Let us present some results in dimensions $N=2$ based on algorithms 1 and 2. \\

 \textbf{ Example 1.} Consider a hinged non-homogeneous plate $\Omega$  which it is made of two different materials. We want to find the solutions of optimization problems \eqref{maxp1}- \eqref{minp1}.
  Set $c_1=1$ and  $c_2=2$, we illustrate the optimum sets in cases rectangle, circle, ellipse and crescent   such that $|\Om|=16.00, 6.28, 16.49, 6.28$ and $S_2=4.00, 3.14, 4.00, 0.78$ respectively.  Remember that  our aim is to locate these two materials  throughout $\Omega$ so to optimize  the first eigenvalue
 in the vibration of the corresponding plate.
 Distribution of these materials as the level sets of the maximal density functions are plotted  in figures \ref{figmax-1} and \ref{figmax-2} for various geometries $\Om$. The sets with the highest density are depicted in black. These shapes reveal that the maximal distribution of the materials consists of putting the material with the highest density in a neighborhood of the boundary. In figures \ref{figmin-1} and   \ref{figmin-2}, distribution of these materials as the level sets of the minimal density functions are plotted.
   Physically speaking, in order  to minimize the basic frequency of the hinged non-homogeneous plates it is best to place the material with the highest density in a region in the center of the domain.

  \begin{figure}[t]\label{figmax-clamp}
  \centering
  \subfloat[$\lambda_{\max}=654.16$]{\label{figmax-clamp:r}\includegraphics[width=0.14\textwidth]{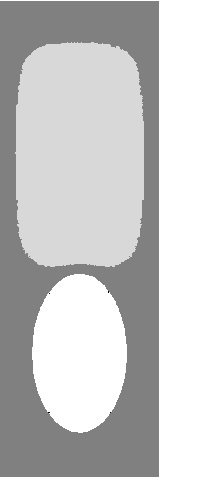}}
  \subfloat[$\lambda_{\max}=6.51$]{\label{figmax-clamp:c}\includegraphics[width=0.22\textwidth]{circle-max.jpg}}
    \caption{The maximizer sets in black}
  \label{figmax-clamp}
\end{figure}

    \begin{figure}[t]\label{figmin-clamp}
  \centering
  \subfloat[$\lambda_{\min}=327.41 $]{\label{figmin-clamp:r}\includegraphics[width=0.14 \textwidth]{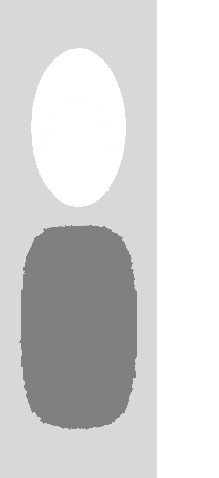}}
  \subfloat[$\lambda_{\min}=3.63 $]{\label{figmin-clamp:c}\includegraphics[width=0.2\textwidth]{circle-min-3.jpg}}
    \caption{The minimizer sets in black}
  \label{figmin-clamp}
\end{figure}

 \begin{figure}[t]\label{threematerials}
  \centering
  \subfloat[$\lambda_{\max}=10.33 $]{\label{threematerials:max}\includegraphics[width=0.177\textwidth]{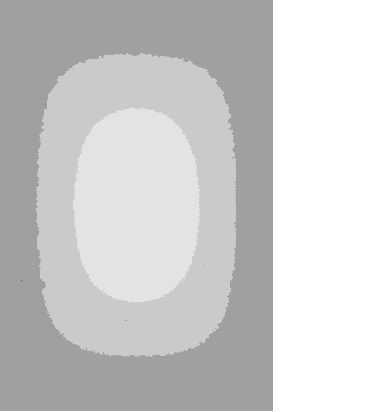}}
  \subfloat[$\lambda_{\min}=4.56 $]{\label{threematerials:min}\includegraphics[width=0.135\textwidth]{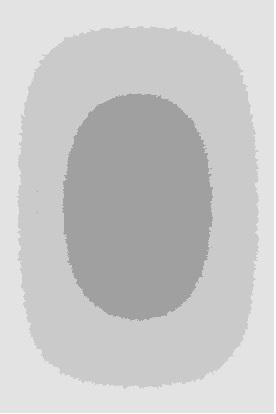}}
    \caption{Rectangular plate with three different materials }
  \label{threematerials}
\end{figure}

   \textbf{ Example 2.} Consider a clamped non-homogeneous plate $\Omega$  that is made of two different materials with densities $c_1=1$ and  $c_2=2$. This means that we should deal with the solutions of optimization problems \eqref{maxp2}- \eqref{minp2}. The optimum sets are illustrated when $\Om$ is a circle and
   a rectangle with a hole   such that $|\Om|=6.28, 2.52   $ and $S_2= 3.14, 1.12 $ respectively. The distribution of these materials as the level sets of the maximal and the minimal density functions are plotted  in figures \ref{figmax-clamp} and \ref{figmin-clamp}. The sets with the highest density are depicted in black. These shapes reveal that the maximal distribution of the materials consists of putting the material with the highest density in a neighborhood of the boundary and the minimal distribution consists of putting the material with the highest density in the center of the region.

 \textbf{ Example 3.} Let $\Om$ be a hinged plate which is made of three different materials with densities $c_1=1$, $c_2=2$ and $c_3=3$. The optimum sets are depicted when $\Om$ is a rectangle with $|\Om|=6$ and $S_i=2, i=1..3$. The locations of these materials in $\Om$ as the level sets of the maximal and minimal density functions are shown in figure \ref{threematerials:max} and \ref{threematerials:min}. According to this optimum density functions, one can discover that the maximal distribution of the materials consists of putting the material with the highest density  around the boundary and the material with the lowest density in the center of $\Om$. In addition, the minimal distribution of the materials consists  of putting the material with the lowest density  around the boundary and the material with the highest density in the center of $\Om$.

\begin{rem}
  In our numerical tests, the procedures typically converge to the global optimal
set of the respective problems, although this has not been
established theoretically that the derived sets are the global optimizers.
It is noteworthy  that    the algorithms may  stick to a local minimizer. This is, the main drawback of such algorithms, see \cite{kao2,lurian}. To overcome this problem, a usual way is to run the algorithms with different initializers . Then, one can compare the derived optimizers and choose the best one.

 When $\Om$ is a ball, the numerical algorithm converges  to the Schwartz increasing rearrangement of $\rho_0$ in the maximization problem. For the minimization problem,  the algorithm  converges  to the  Schwartz decreasing rearrangement of $\rho_0$. Indeed, both procedures converge to the optimal solutions derived analytically in \cite{cuccu,cuccu2}.
 Of course, our results derived in the above examples agree with  physical intuition.

\end{rem}



\begin{thebibliography}{99}
\bibitem{agmon} S. Agmon, A. Douglis, L. Nirenberg, Estimates near the boundary for solutions of elliptic partial differential equations
satisfying general boundary conditions I, { Commun. Pure Appl. Math.} { 12}  (1959)  623--727.


\bibitem{alliare} G. Alliare, {Shape optimization by the homogenization method in applied mathematical sciences}, vol. 145, Springer-Verlag, New York, 2002.

\bibitem{Alvino} A. Alvino, G. Trombetti, P. -L. Lions,  { On optimization problems with prescribed rearrangements,}
Nonlinear Anal. {13} (1989) 185--220.

\bibitem{anedda} C. Anedda, F. Cuccu, G. Porru, Minimization of the first eigenvalue in
problems involving the bi-laplacian, { Revista de Matem\'{a}tica: Teor\'{\i}a Y Aplicaciones}, 16 (2009) 127--136.

\bibitem{bu89} G.R. Burton, { Variational problems on classes of rearrangements and multiple configurations for steady
vortices, } Ann. Inst. H. Poincaré. Anal. Non Linéaire 6. {4}
(1989) 295--319.

\bibitem{chanillo}S. Chanillo, D. Grieser, M. Imai, K. Kurata and I. Ohnishi, {Symmetry breaking and other
phenomena in the optimization of eigenvalues for composite
membranes}, Commun. Math. Phys. 214 (2000) 315--337.

\bibitem{lurian} C. Conca, A. Lurian, R. Mahadevan, {  Minimization of the ground state for two phase conductors in low contrast
regime,}  Siam J. Appl. Math. {72} (2012) 1238--1259.


\bibitem{conca} C. Conca, R. Mahadevan, L. Sanz, { A minimax principle
for nonlinear eigenproblems depending continuously on the
eigenparameter,}  Appl. Math. Optim. {60} (2009)  173--184.


\bibitem{cox90}S.J. Cox, J.R. McLaughlin, {Extremal eigenvalue problems for composite membranes, I, II},
Appl. Math. Optim. 22 (1990) 153--167, 169--187.


\bibitem{cuccu} F. Cuccu, G. Porru, Maximization of the first eigenvalue in problems involving the bi-Laplacian, { Nonlinear Anal.} { 71}  (2009)  800--809.

\bibitem{cuccu2} F. Cuccu, G. Porru, Optimization of the first eigenvalue in problems involving the bi-Laplacian, { Differ. Equ. Appl.} { 1}  (2009)  219--235.

 \bibitem{cuccu2011}  F. Cuccu, G. Porru,  S. Sakaguchi, Optimization problems on general classes of rearrangements, J. Math. Anal. Appl. 74  (2011) 5554--5565.

\bibitem{derlet}A. Derlet, J.-p. Gossez, P. Tak\'{a}\v{c}, {Minimization of eigenvalues for a quasilinear elliptic Nuemann problem with indefinite weight},
J. Math. Anal. Appl. 371  (2010) 69--79.
\bibitem{gilbarg}D. Gilbarg, N.S. Trudinger, {Elliptic partial differential equations of second order, } second edt, Springer-Verlag,
New York, 1998.

\bibitem{hekao} L. He, C.-Y. Kao, S. Osher,  {Incorporating topological derivatives into shape derivatives based level set methods, }  J. Comp. Phys. 225 (2007) 891--909.

\bibitem{henrot} A. Henrot, {Extremum problems for eigenvalues of elliptic operators, } { Birkh\"{a}user-Verlag, Basel}, 2006.
\bibitem{kao1} M. Hinterm�ller, C.-Y. Kao,  A. Laurain , { Principal eigenvalue minimization for an elliptic problem with indefinite weight
 and robin boundary conditions,} Appl. Math. Optim. {65} (2012) 111--146.

\bibitem{kaoluo} C.-Y. Kao, Y.Lou, E. Yanagida,  {Principal eigenvalue for an elliptic  problem with indefinite weight on cylindrical domains, }  Mathematical Bioscience and Engineering. 5 (2008) 315--335.


\bibitem{kaoyab} C.-Y. Kao, S. Osher, E. Yablonovitch,  {Maximizing band gaps in two dimensional photonic crystals by using level set methods, }  Appl. Phys. B-Laser. O. 81 (2005) 235--244.
  \bibitem{kaosantosa2} C.-Y. Kao, F. Santosa,  {Maximization of the quality factor of an optical resonator, }  Wave Motion. 45 (2008) 412--427.


    \bibitem{kao2} C.-Y. Kao, S. Su , { An efficient rearrangement algorithm for shape optimization on eigenvalue problems,}  J. Sci. Comput. 54  (2013) 492--512.

\bibitem{kong} X. Kong, Z. Huang , { A way of updating the density function for the design of the drum,}  Comput. Math. Appl. 66  (2013) 62--80.
\bibitem{kurata} K. Kurata,  M. Shibata, S. Sakamoto, { Symmetry-breaking phenomena in an optimization problem
for some nonlinear elliptic equation,}  Appl. Math. Optim. {50} (2004)  259--278.

\bibitem{abbasali} A. Mohammadi, F. Bahrami, H. Mohammadpour, {  Shape dependent energy optimization in quantum dots,}  Appl. Math. Lett. {25} (2012) 1240--1244.

\bibitem{abbasali2} A. Mohammadi, F. Bahrami,   {  A nonlinear eigenvalue problem arising in  a nanostructured quantum dot,}  submitted.

\bibitem{oshersantosa} J. Osher, F. Santosa, {Level set methods for optimization problems involving geometry and constraints i. frequencies of a two-density inhomogeneous drum, }  J. Comp. Phys. 171 (2001) 272--288.


\bibitem{osher}  S. Osher, J.A. Sethian, {Fronts propagating with curvature-dependent speed: algorithms based on Hamilton-Jacobi formulations, } J. Comput. Phys. 79 (1988) 12--49.


\bibitem{stru} M. Struwe, {Variational methods, } { Springer-Verlag, Berlin, New York}, 1990.

\bibitem{tartar} L. Tartar, {Compensated compactness and applications to partial
differential equations, Nonlinear analysis and mechanics},
Heriot-Watt Symposium, Vol. IV, 136- 212, Res. Notes in Math., 39,
Pitman, Boston, Mass.-London, 1979.




\bibitem{zhu} S. Zhu, Q. Wu, C. Liu,  {Variational  piecewise constant level set methods for shape optimization of a two-density drum, }  J. Comp. Phys. 229 (2010) 5062--5089.
















\end{thebibliography}
\end{document}